\documentclass[a4paper]{article}
\usepackage{amsmath,amssymb}
\usepackage[colorlinks,citecolor=blue]{hyperref}
\usepackage[utf8]{inputenc}
\usepackage{tikz}
\usepackage[english]{babel}
\usepackage[capitalize]{cleveref}
\usepackage{ stmaryrd }

\newtheorem{theorem}{Theorem}[section]

\newtheorem{lemma}[theorem]{Lemma} 

\newtheorem{remark}[theorem]{Remark}

\newenvironment{proof}{\begin{trivlist}\item{\bf{Proof.}}}
  {\hfill\rule{2mm}{2mm}\end{trivlist}}

\newcommand{\NN}{\mathbb{N}}
\newcommand{\CC}{\mathbb{C}}
\newcommand{\QQ}{\mathbb{Q}}
\newcommand{\RR}{\mathbb{R}}

\def\x{{\bf x}}
\def\a{{\bf a}}
\def\s{{\bf s}}
\newcommand{\alphab}{\boldsymbol{\alpha}}

\title{A note on Dirichlet-like series attached to polynomials}
\author{F. Chapoton}
\date{\today}
\begin{document}

\maketitle

\begin{abstract}
  Some Dirichlet-like functions, attached to a pair (periodic
  function, polynomial) are introduced and studied. These functions
  generalize the standard Dirichlet $L$-functions of Dirichlet
  characters. They have similar properties, being holomorphic on the
  full complex plane and having simple values on negative integers.
\end{abstract}

\section*{Introduction}

In this article, we study functions of a complex variable $s$ defined
by the series
\begin{equation}
  \label{type_d}
  L_{\chi,P}(s) = \sum_{n \geq 1} \chi(n)\frac{P'(n)}{{P(n)}^s},
\end{equation}
where $\chi$ is a Dirichlet character or a periodic function and $P$
is a polynomial in $X$ that does not vanish at positive
integers. Similar but simpler functions, defined by
\begin{equation}
  \label{type_z}
  Z_P(s) = \sum_{n \geq 1} \frac{P'(n)}{{P(n)}^s},
\end{equation}
were considered in a previous work~\cite{chap_essouabri}. It was shown
there that these functions can be extended to meromorphic functions on
$\CC$ with just a simple pole at $1$.

The functions in~\eqref{type_z} can be considered as somewhat exotic
variants of the Riemann zeta function $\zeta$, and they share some of
the properties of $\zeta$, which is the special case when $P(X)=X$. In
particular, they have a simple pole at $s=1$ and there is a uniform
formula for their values at negative integers, involving Bernoulli numbers.

Similarly, the functions in~\eqref{type_d} are exotic variants of the
Dirichlet $L$-functions, which are included as the special case when
$P(X)=X$. Our aim is to prove that they also share some properties
with them. Assuming always that the periodic function $\chi$ has mean
value zero over a period, we will prove that they extend to
holomorphic functions on the complex plane. We also obtain a uniform
formula for their values at negative integers.

\medskip

The original motivation for the study of the functions
in~\eqref{type_z} came from two specific elements $A$ and $\Omega$ of
an algebra of formal tree-indexed series over $\QQ$, namely the
completion of the free pre-Lie algebra on one generator. The
tree-indexed series $A$ and $\Omega$ are inverse of each other for
composition, so that $A$ is like an exponential and $\Omega$ is like a
logarithm. The coefficients of $\Omega$ are rational numbers, among
which all the Bernoulli numbers, that turn out to be values at
negative integers of functions as defined in~\eqref{type_z}. For more
on that, see~\cite{chap_q_omega}.

\medskip

The results of the present article have already played a role as the
central motivation in~\cite{BostanChap}, where we considered the
ordinary generating series of the values at negative integers of
$L_{\chi,P}$ and their Jacobi continued fractions, for simple periodic
functions $\chi$ and in the case of quadratic polynomials $P$.

\medskip

Instead of working only with Dirichet characters, we will consider for
$\chi$ any periodic function. We first recall known results on
standard Dirichlet $L$-functions for such $\chi$ in
\Cref{dirichlet}. In \cref{heuristic}, a short but heuristic reason is
given for the formula expressing values of $L_{\chi,P}$ at negative
integers using a linear form applied to the powers of $P$.  In
\Cref{holomorph}, the main results are stated and proved.  In
\Cref{observe}, an experimental observation about congruences modulo
prime numbers is reported.


Note that the proofs are essentially the same as those in~\cite{chap_essouabri}, once adapted by the appropriate insertion of
the function $\chi$.


{\bf keywords:} Dirichlet character; periodic function; Dirichlet-like series; analytic continuation; Bernoulli number



\section{Dirichlet characters and periodic coefficients}\label{dirichlet}

Let us recall classical results on Dirichlet characters, or rather on all periodic functions.

We refer to the book~\cite[\S9 \& \S10]{cohen_livre_2} for the theory
of Dirichlet's $L$-functions and the description of their values at
negative integers. In particular, section~10.2.1 of this book contains
general statements about Dirichlet series with periodic coefficients.

Let $\chi$ be a function from $\NN^*$ to $\CC$, periodic of period
$N$ and of zero sum over a period, meaning that
\begin{equation*}
  \sum_{n=1}^{N} \chi(n) = 0.
\end{equation*}

 
The Dirichlet series for the periodic function $\chi$ is defined for $s$ in the right half-plane $\{\Re(s) > 1\}$ by the series
\begin{equation*}
  L_\chi(s) = \sum_{n \geq 1} \chi(n) \frac{1}{n^s}.
\end{equation*}
It extends uniquely to a meromorphic function in the entire complex
plane, with a possible pole at $1$. The one-period zero-sum assumption
implies that the resulting function is holomorphic.

Let us recall the known description of the values taken at negative integers by the function $L_\chi$, reformulating slightly some well-known statements.

We associate to $\chi$ the linear form $\Psi_\chi$ on the space of polynomials in $X$ defined by the formula
\begin{equation}
  \label{defi_Psi}
  t \frac{\sum_{n=1}^{N} \chi(n) e^{n t}}{1-e^{N t}} = - \sum_{n\geq 0} \Psi_\chi(X^n)\frac{t^n}{n!} = - \Psi_\chi(e^{X t}),
\end{equation}
as an equality between power series in the variable~$t$, where $N$ is the period of~$\chi$. Then, for all $m\geq 1$, we have the relation
\begin{equation}
  \label{eval_L_chi}
  L_\chi(1-m) = - \frac{\Psi_\chi(X^m)}{m}.
\end{equation}

The values $\Psi_\chi(X^m)$ for $m\geq 0$ are the analogues for the
function $\chi$ of the Bernoulli numbers. The assumption that $\chi$
has zero-sum over a period implies that $\Psi_\chi(1)=0$.

\section{$L$-functions of polynomials}\label{heuristic}

\medskip

Let $\chi$ be a periodic function with period $N$ and zero sum over a
period. Let $P$ be a polynomial in $\RR[X]$ with positive leading
coefficient and that does not vanish at positive integers. We
introduce the convergent series
\begin{equation}
  \label{def_pseudo_L}
  L_{\chi, P}(s) = \sum_{n\geq 1} \chi(n) \frac{P'(n)}{{P(n)}^s}
\end{equation}
for complex numbers $s$ with $\Re(s)>1$. These are not Dirichlet
series in the usual sense, and they certainly do not have an Euler
product expression, unless $P$ is a monomial.

We will show in the next section that the function $L_{\chi, P}(s)$ extends to a holomorphic
function of the complex parameter $s$. Moreover its values at negative
integers are given in terms of $\Psi_{\chi}$ and $P$ by the formula
\begin{equation}
  \label{negative_value}
  L_{\chi, P}(1-m) = -\frac{\Psi_{\chi}(P^m)}{m}
\end{equation}
for all $m \in \NN^*$.

Before giving a formal proof, let us present a heuristic
argument. One can deduce from~\eqref{defi_Psi} that
\begin{equation*}
   \Psi_{\chi}(E(N+X)) - \Psi_{\chi}(E) =  \sum_{n=1}^{N} \chi(n) E'(n),
\end{equation*}
for every polynomial $E$. After a telescoping summation, one gets
\begin{equation*}
  \Psi_{\chi}(E(\ell N +X)) - \Psi_{\chi}(E) =  \sum_{n=1}^{\ell N} \chi(n) E'(n).
\end{equation*}
For polynomials that are powers (of the shape $P^m$ for some $P$), one
therefore gets
\begin{equation*}
   \Psi_{\chi}({P(\ell N +X)}^m) - \Psi_\chi(P^m) =  m \sum_{n=1}^{\ell N } \chi(n) P'(n) {P(n)}^{m-1}.
\end{equation*}
Formally going to the limit $\ell=\infty$ (and assuming that the first
term of the left-hand side disappears) gives formula~\eqref{negative_value}.

\section{Study of Dirichlet-like functions}\label{holomorph}

For all $\x=(x_1,\dots, x_d) \in \RR^d$ and all $\alphab =(\alpha_1,\dots, \alpha_d) \in \NN^d$, we will use in the sequel the following notations:
$ \x^{\alphab} =x_1^{\alpha_1}\dots x_d^{\alpha_d}$ and $ |\alphab|=\alpha_1+\cdots +\alpha_d$. We denote also for any $z\in \CC$ verifying $\Re z >0$ and any $s\in \CC$, $z^s=e^{s \log x}$ where $\log$ is the principal determination of the logarithm.\par

The purpose of this section is to prove the following result:
\begin{theorem}\label{prolong}
  Let $P \in \RR[X]$ be a polynomial of degree $d\geq 1$ with positive
  leading coefficient. Let $a_1,\dots, a_d\in \CC$ be the roots (not
  necessarily distinct) of $P$.  Let $A\in \NN^*$ be such that
  $\forall a \geq A \enskip \Re P(a) > 0$.  We consider the Dirichlet-like
  series
 \begin{equation*}
   L_{A, \chi, P}(s):= \sum_{n=A}^{+\infty} \chi(n) \frac{P'(n)}{{P(n)}^s}.
\end{equation*}
 Then:
 \begin{enumerate} 
 \item The function $s\mapsto L_{A,\chi,P}(s)$ converges absolutely in the half-plane $\{ \Re(s) >1\}$ and has a holomorphic continuation to the whole complex plane $\CC$;
 \item for any $m\in \NN^*$, $ \quad L_{A,\chi,P}(1-m)= -\frac{1}{m} \Psi_{\chi} (P^m) -\sum_{n=1}^{A-1} \chi(n) P'(n) {P(n)}^{m-1} $.
\end{enumerate}
\end{theorem}

\begin{remark}
  By taking $A=1$, point 2 gives the formula~\eqref{negative_value}.
\end{remark}

We need the following elementary lemma, proved in~\cite[Lemma 6]{chap_essouabri}:
\begin{lemma}\label[lemma]{taylorplus}
 Let $d\in \NN^*$ and $\a=(a_1,\dots, a_d)\in \NN^d\setminus \{(0,\dots, 0)\}$. Set $\delta ={\left(2 \max_j |a_j|\right)}^{-1} >0$. Then, for any $N\in \NN$, any $\s=(s_1,\dots, s_d) \in \CC^d$ and any $x \in [-\delta, \delta]$, one has 
 \begin{equation}\label{taylorrelation}
   \prod_{j=1}^d {(1 - x a_j)}^{-s_j} = \sum_{\ell=0}^N c_{\ell} (\s)~x^{\ell} + x^{N+1} \rho_N (x; \s)
 \end{equation}
 where
 \begin{equation*}
   c_{\ell}(\s) = {(-1)}^{\ell} \sum_{\substack{\alphab \in \NN^d \\ |\alphab| =\ell }} \a^{\alphab} \prod_{j=1}^d \binom{-s_j}{\alpha_j} 
 \end{equation*}
  and
 \begin{equation*}
 \rho_N (x; \s) = {(-1)}^{N+1} (N+1) \sum_{\substack{\alphab \in \NN^d \\ |\alphab| = N+1}} \a^{\alphab} \prod_{j=1}^d \binom{-s_j}{\alpha_j} 
 \int_0^1 {(1-t)}^N \prod_{j=1}^d {(1- t x a_j)}^{-s_j-\alpha_j} ~dt.
 \end{equation*}
 Moreover one has:
 \begin{enumerate} 
 \item for any $x \in [-\delta, \delta]$, $\s\mapsto \rho_N (x ; \s) $ is holomorphic in the whole space $\CC^d$;
 \item for any compact subset $K$ of $\CC^d$, there exists a constant $C=C(K,\a,N,d) >0$ such that 
 \begin{equation*}
   \forall (x,\s) \in [-\delta,\delta] \times K \qquad |\rho_{N} (x ;\s)|\leq C. \end{equation*}
 \end{enumerate}
\end{lemma}

\begin{proof}
  {\bf (point 1 of \Cref{prolong})}\\
  For short, let us write $L(s)$ for $L_{A,\chi,P}(s)$.

  First we remark that if $\a=(a_1,\dots, a_d)=(0,\dots, 0)$, then $P$
  is of the form $P(X) = c X^d$ where $c > 0$.  It follows that
  $L(s)=d c^{1-s} L_\chi (d s-d+1)$ and therefore point 1 of
  \Cref{prolong} holds in this case, by the known properties of these
  Dirichlet L-functions for periodic coefficients, as recalled in
  \Cref{dirichlet}.

  We will assume in the sequel that $\a \neq (0,\dots,0)$ and set $\delta ={\left(2 \max_j |a_j|\right)}^{-1} >0$.
  We will note in the sequel $s=\sigma+ i\tau$ where $\sigma =\Re (s)$ and $\tau =\Im (s)$.
  It is easy to see that
\begin{equation*}
  \left|\chi(n) \frac{P'(n)}{{P(n)}^s}\right| \ll \frac{1}{n^{d\sigma -(d-1)}}.
\end{equation*}
It follows that $s\mapsto L(s)$ converges absolutely in the half-plane $\{ \Re(s) >1\}$.\\

As the act of removing or adding a finite number of terms does not
change the holomorphicity, we can choose the integer $ A $ as
large as needed.
Let us choose here $A \in \NN^*$ such that $A \geq 2 \max_{j} |a_j| = \delta^{-1}$.\\
It is clear that we can also assume without loss of generality that
the polynomial $P$ is unitary. It follows that
\begin{equation*}
  P(X)= \prod_{j=1}^d (X-a_j) \quad {\mbox { and }} \quad P'(X) = P(X) \left( \sum_{j=1}^d \frac{1}{X-a_j}\right).
\end{equation*}
We deduce that for all $s \in \CC$ satisfying $\sigma = \Re (s)> 1$ there holds:
\begin{eqnarray*}
L(s)&=& \sum_{n=A}^{+\infty} \chi(n) \frac{P'(n)}{{P(n)}^s}
= \sum_{j=1}^d \sum_{n=A}^{+\infty} \frac{\chi(n)}{{(n-a_j)}^s \prod_{k\neq j} {(n - a_k)}^{s-1}}\\
&=& \sum_{j=1}^d \sum_{n=A}^{+\infty} 
\frac{\chi(n)}{n^{ds -(d-1)}} {\left(1-\frac{a_j}{n} \right)}^{-s} \prod_{k\neq j} {\left(1-\frac{a_k}{n} \right)}^{-s+1}.
\end{eqnarray*}
Let $N\in \NN$. \Cref{taylorplus} with $x=1/n$ and the previous relation imply that for all $s\in \CC$ verifying $\sigma =\Re(s) >1$ we have:
\begin{equation*}
L(s)= \sum_{j=1}^d \sum_{n=A}^{+\infty} 
\frac{\chi(n)}{n^{ds -(d-1) }} \left[\sum_{\ell=0}^N c_{\ell}\left(f_j(s)\right) \frac{1}{n^{\ell}} + \frac{1}{n^{N+1}} \rho_N \left(1/n ; f_j(s)\right)\right],
\end{equation*} 
where $f_j(s) =(s_1,\dots,s_d)$ with $s_k =s-1$ if $k\neq j$ and $s_j=s$.\\

We deduce that for all $s\in \CC$ verifying $\sigma =\Re(s) >1$ there holds:
\begin{eqnarray}\label{bon1}
L(s)&=& \sum_{\ell=0}^N \left[\sum_{j=1}^{d} c_{\ell}\left(f_j(s)\right)\right] L_{A,\chi}\left(d s -(d-1)+\ell\right) \nonumber \\
& & + \sum_{n=A}^{+\infty} 
\frac{\chi(n)}{n^{ds -(d-1)+N+1}} \left[\sum_{j=1}^d \rho_N \left(1/n ; f_j(s)\right)\right],
\end{eqnarray}
where $L_{A,\chi}(s):= \sum_{n=A}^{+\infty} \frac{\chi(n)}{n^s} = L_\chi(s) - \sum_{n=1}^{A-1} \frac{\chi(n)}{n^s}$.\\

Moreover,
\begin{enumerate}
\item the point 2 of \Cref{taylorplus} and the dominated convergence theorem of Lebesgue imply that 
\begin{equation*}s\mapsto \sum_{n=A}^{+\infty} 
\frac{\chi(n)}{n^{ds -(d-1)+N+1}} \left[\sum_{j=1}^d \rho_N \left(1/n ; f_j(s)\right)\right]\end{equation*}
is defined and is holomorphic in the half-plane $\{\sigma > 1-\frac{N+1}{d}\}$;
\item the classical properties of the Dirichlet $L$-functions imply that the function $s\mapsto L_{A,\chi}(s)$ is holomorphic in the whole complex plane $\CC$.
\end{enumerate}
These last two points and identity~\eqref{bon1} implies that $s \mapsto L(s) $ has a holomorphic extension to the half-plane $\{\sigma > 1-\frac{N+1}{d}\}$.
As $N\in \NN$ is arbitrary, we deduce that $s\mapsto L(s)$ has a holomorphic continuation to the whole complex plane $\CC$.
\end{proof}

Before proving point 2 of \Cref{prolong}, let us do a preliminary computation, using the notations of \Cref{taylorplus}. It is easy to see that for all $\ell\in \NN$:
\begin{eqnarray}\label{bon2}
\sum_{j=1}^d c_{\ell}\left(f_j(s)\right) &=& \sum_{j=1}^d{(-1)}^{\ell} \sum_{\substack{\alphab \in \NN^d \\ |\alphab| =\ell }} \a^{\alphab} \binom{-s}{\alpha_j} \prod_{k\neq j} \binom{-s+1}{\alpha_k} \nonumber \\
&=& \sum_{j=1}^d {(-1)}^{\ell} \sum_{\substack{\alphab \in \NN^d \\ |\alphab| =\ell }} \a^{\alphab} \frac{s+\alpha_j-1}{s-1} \prod_{k=1}^d \binom{-s+1}{\alpha_k} \nonumber \\
&=& {(-1)}^{\ell} \sum_{\substack{\alphab \in \NN^d \\ |\alphab| =\ell }} \a^{\alphab} \prod_{k=1}^d \binom{-s+1}{\alpha_k} \sum_{j=1}^d \frac{s+\alpha_j-1}{s-1} \nonumber \\
&=& {(-1)}^{\ell} \sum_{\substack{\alphab \in \NN^d \\ |\alphab| =\ell }} \a^{\alphab} \prod_{k=1}^d \binom{-s+1}{\alpha_k} \frac{d s-d + \ell}{s-1}.
\end{eqnarray}
Relations~\eqref{bon1} and~\eqref{bon2} imply that for all $s\in \CC$ satisfying $\sigma =\Re(s) >1$ we have:
\begin{eqnarray}\label{bon3}
(s-1) L(s)&=& \sum_{\ell=0}^N \left[ {(-1)}^{\ell} \sum_{\substack{\alphab \in \NN^d \\ |\alphab| =\ell }} \a^{\alphab} \prod_{k=1}^d \binom{-s+1}{\alpha_k}\right] \left(d s-d + \ell\right) 
L_{A,\chi}\left(d s -(d-1)+\ell\right) \nonumber \\
& & + (s-1) \sum_{n=A}^{+\infty} 
\frac{\chi(n)}{n^{ds -(d-1)+N+1}} \left[\sum_{j=1}^d \rho_N \left(1/n ; f_j(s)\right)\right].
\end{eqnarray}

\begin{proof}
  {\bf (point 2 of \Cref{prolong})}\\
  First let us recall from~\eqref{eval_L_chi} the formula 
  \begin{equation}\label{zetanegative}
    m ~L_\chi (1-m) = - \Psi_\chi(X^m) \quad \forall m\in \NN^*.
  \end{equation}
  Let $m \in \NN^*$. Set $N= d m$. In particular, $1 - m >
  1-\frac{N+1}{d}$. If $|\alphab|=N+1$, then for any $j=1,\dots, d$:
  \begin{equation*}
    \alpha_j >m-1 \quad {\mbox { or }} \quad {\mbox { there exists }} k \in \{1,\dots d\}\setminus \{j\} {\mbox { such that }} \alpha_k >m.
  \end{equation*}
  We deduce that for any $j=1,\dots, d$ and any $x$: 
  \begin{eqnarray*}
    \rho_N \left(x; f_j(1-m)\right) &=& {(-1)}^{N+1} (N+1) \sum_{\substack{\alphab \in \NN^d \\ |\alphab| = N+1}} \a^{\alphab}\binom{m-1}{\alpha_j} \prod_{k\neq j}^d \binom{m}{\alpha_k} \\
    & & \times 
    \int_0^1 {(1-t)}^N {(1- t x a_j)}^{m-\alpha_j} \prod_{k\neq j}^d {(1- t x a_k)}^{m-1-\alpha_k} ~dt\\
    &=& 0.
  \end{eqnarray*}
  It follows then from~\eqref{bon3} that 
  \begin{eqnarray}\label{neg1}
    -m L(1-m)&=& \sum_{\ell=0}^N \left[ {(-1)}^{\ell} \sum_{\substack{\alphab \in \NN^d \\ |\alphab| =\ell }} \a^{\alphab} \prod_{k=1}^d \binom{m}{\alpha_k}\right] \left(\ell-d m\right) 
    L_{A,\chi}\left(1+\ell-d m\right) \nonumber \\
    &=& \sum_{\ell=0}^N \left[ {{(-1)}^{\ell}} \sum_{\substack{\alphab \in \NN^d \\ |\alphab| =\ell }} \a^{\alphab} \prod_{k=1}^d \binom{m}{\alpha_k}\right] \left(\ell-d m\right) 
    L_{\chi} \left(1+\ell-d m\right) \nonumber\\
    & & -\sum_{\ell=0}^N \left[ {{(-1)}^{\ell}} \sum_{\substack{\alphab \in \NN^d \\ |\alphab| =\ell }} \a^{\alphab} \prod_{k=1}^d \binom{m}{\alpha_k}\right] \left(\ell-d m\right) 
    \left(\sum_{n=1}^{A-1} \chi(n) n^{d m -\ell -1}\right).
  \end{eqnarray}
  This sum therefore splits into two parts. Remarking that if
  $|\alphab| > N$ then there exists $k$ such that $\alpha_k > m$ and
  hence $ \binom{m}{\alpha_k}=0$, we can compute the second part:
  \begin{eqnarray}\label{neg2}
    \kappa&:=& -\sum_{\ell=0}^N \left[ {(-1)}^{\ell} \sum_{\substack{\alphab \in \NN^d \\ |\alphab| =\ell }} \a^{\alphab} \prod_{k=1}^d \binom{m}{\alpha_k}\right] \left(\ell-d m\right) 
    \left(\sum_{n=1}^{A-1} \chi(n) n^{d m -\ell -1}\right)\\
    &=&\sum_{n=1}^{A-1} \chi(n) n^{d m -1} \sum_{\ell=0}^N \sum_{\substack{\alphab \in \NN^d \\ |\alphab| =\ell }} (d m -\sum_{j=1}^d \alpha_j) 
    \left[ \prod_{k=1}^d \binom{m}{\alpha_k}{\left(-\frac{a_k}{n}\right)}^{\alpha_k}\right]\nonumber\\
&=& \sum_{n=1}^{A-1} \chi(n) n^{d m -1} \sum_{\alphab \in {\{0,\dots, m\}}^d} \ (d m -\sum_{j=1}^d \alpha_j) 
\left[ \prod_{k=1}^d \binom{m}{\alpha_k}{\left(-\frac{a_k}{n}\right)}^{\alpha_k}\right].\nonumber
\end{eqnarray}
Continuing this computation by splitting this sum in two, we have
\begin{eqnarray}\label{neg3}
\kappa&=& d m \sum_{n=1}^{A-1} \chi(n) n^{d m -1} 
\prod_{k=1}^d {\left(1-\frac{a_k}{n}\right)}^{m}
 -\sum_{n=1}^{A-1} \sum_{j=1}^d \chi(n) n^{d m -1} \frac{m\left(-\frac{a_j}{n}\right)}{1-\frac{a_j}{n}}\prod_{k=1}^d {\left(1-\frac{a_k}{n}\right)}^{m}\nonumber\\
&=& d m \sum_{n=1}^{A-1} n^{-1} \chi(n) {P(n)}^m+ m \sum_{n=1}^{A-1} \sum_{j=1}^d \chi(n) \frac{a_j}{ n(n-a_j)} {P(n)}^m \nonumber\\
&=& m \sum_{n=1}^{A-1} {\chi(n) P(n)}^{m-1} P'(n).
\end{eqnarray}

Relations~\eqref{zetanegative},~\eqref{neg1},~\eqref{neg2} and~\eqref{neg3} imply that 
\begin{equation}\label{fin1}
(-m) L(1-m)= \sum_{\ell=0}^N \left[ {(-1)}^\ell \sum_{\substack{\alphab \in \NN^d \\ |\alphab| =\ell }} \a^{\alphab} \prod_{k=1}^d \binom{m}{\alpha_k}\right] \Psi_\chi(X^{d m -\ell}) + m \sum_{n=1}^{A-1} \chi(n) {P(n)}^{m-1} P'(n).
\end{equation}
On the other hand, it is easy to see that 
\begin{eqnarray*}
{P(X)}^m &=& \prod_{j=1}^d {(X-a_j)}^m= \prod_{j=1}^d \left( \sum_{\alpha_j =0}^m \binom{m}{\alpha_j} {(-a_j)}^{\alpha_j} X^{m-\alpha_j}\right)\\
&=& \sum_{\alphab \in {\{0,\dots, m\}}^d} {(-1)}^{|\alphab|} \a^{\alphab} \left(\prod_{j=1}^d \binom{m}{\alpha_j}\right) X^{d m -|\alphab|}
= \sum_{\ell=0}^N {(-1)}^{\ell} \left[\sum_{\substack{\alphab \in \NN^d \\ |\alphab|=\ell}} \a^{\alphab} \prod_{j=1}^d \binom{m}{\alpha_j}\right] X^{d m -\ell}.
\end{eqnarray*}
It follows that 
\begin{equation*}\Psi_\chi({P(X)}^m)= \sum_{\ell=0}^N \left[{(-1)}^{\ell} \sum_{\substack{\alphab \in \NN^d \\ |\alphab|=\ell}} \a^{\alphab} \prod_{j=1}^d \binom{m}{\alpha_j}\right] \Psi(X^{d m -\ell}).\end{equation*}
We then deduce from~\eqref{fin1} that 
$L(1-m)= -\frac{1}{m} \Psi_\chi({P(X)}^m) -\sum_{n=1}^{A-1} \chi(n) {P(n)}^{m-1} P'(n)$. This completes the proof of \Cref{prolong}.
\end{proof}

\begin{remark}
Consider the element $e^{Xt}$ in the ring $\QQ[X]\llbracket t \rrbracket$. This is not a
polynomial in $X$, but one can still formally apply point 2 of \Cref{prolong} with $m=1$, $A=1$ and $P = e^{Xt}$. This gives the defining equation~\eqref{defi_Psi}
for $\Phi_\chi$.
\end{remark}

\section{Example and observation for $\chi_3$ and $X(X+U)$}\label{observe}

Let $\chi_3$ be the unique primitive Dirichlet character of conductor $3$. It takes values $1,-1,0$ on $1,2,3$. Let
\begin{equation*}
  L_3(s) = \sum_{n \geq 1} \chi(n) n^{-s}
\end{equation*}
be the associated Dirichlet $L$-function, holomorphic on $\CC$.

Because the character $\chi_3$ has values in $\QQ$, values of $L_3$ at
negative integers are rational numbers. For all primes $p > 3$, these
values satisfy congruences modulo $p$ that allow to define a $p$-adic
$L$-function by $p$-adic interpolation. A standard reference
about this is~\cite{cohen_livre_2}.

In this special case, let us denote by $\Psi_3$ the associated linear
form on $\QQ[X]$, defined by
\begin{equation*}
  \Psi_3(X^m) = -m L_3(1-m)
\end{equation*}
for $m \geq 0$. By~\eqref{defi_Psi}, the exponential generating series
of values of $\Psi_3$ is the odd series
\begin{equation*}
  -t \frac{e^t - e^{2t}}{1-e^{3t}} = -\frac{1}{3} t + \frac{2}{3} \frac{t^{3}}{3!} - \frac{10}{3} \frac{t^{5}}{5!} + \frac{98}{3} \frac{t^{7}}{7!} - \frac{1618}{3} \frac{t^{9}}{9!} + \frac{40634}{3} \frac{t^{11}}{11!} - \frac{1445626}{3} \frac{t^{13}}{13!} + \cdots
\end{equation*}

With this definition, for $m > 0$,
\begin{equation*}
  L_3(1-m) = - \Psi_3(X^m) / m.
\end{equation*}

One can then use the knowledge of $\Psi_3$ to compute values of any
function $L_{\chi_3,P}$ at negative integers. For example, consider
$P=X(X+1)$. Then
\begin{equation*}
  L_{\chi_3,P}(-1) = - \frac{1}{2} \Psi_3(X^4+2 X^3+X^2)\\
  = - \Psi_3(x^3) = -2/3.
\end{equation*}

\medskip

Let us now introduce a new variable $u$ and define polynomials $p_m$ in
$u$ for $m \geq 1$ by the formula
\begin{equation*}
  p_m(u) = \Psi_3\bigl( {(X(X+u))}^m \bigr) / m.
\end{equation*}
By our general results, these polynomials are the opposite of values of
\begin{equation}
  L_{3,X(X+u)}(s) = \sum_{n \geq 1} \chi_3(n) \frac{ 2 n + u}{{(n(n+u))}^s}
\end{equation}
at negative integers. For all $m\geq 0$, the polynomial $p_m(u)$ is odd because $\Psi_3(X^m)$ always vanishes when $m$ is even.

The first few polynomials, for $m \geq 1$, are
\begin{align*}
 p_1 = \frac{1}{3} u, \quad
 p_2 = -\frac{2}{3} u, \quad
 p_3 = -\frac{2}{9} u^{3} + \frac{10}{3} u,\quad
 p_4 = \frac{10}{3} u^{3} - \frac{98}{3} u,\\
 p_5 = \frac{2}{3} u^{5} - \frac{196}{3} u^{3} + \frac{1618}{3} u,\quad
 p_6 = -\frac{98}{3} u^{5} + \frac{16180}{9} u^{3} - \frac{40634}{3} u,\\
 p_7 = -\frac{14}{3} u^{7} + 1618 u^{5} - \frac{203170}{3} u^{3} + \frac{1445626}{3} u.
\end{align*}

When replacing $u$ by an element of some finite field $\mathbf{F}_p$
for $p > 3$, one observes congruences similar to those for the values
at negative integers of the standard $L_3$ function, as follows.

For instance, assuming $u^5=u$ and mapping all coefficients to
$\mathbf{F}_5$, one gets
\begin{multline*}
  {\color{red} 3 u}, {\color{blue} 4 u, 3 u^{3}, u, 2 u^{3}}, 4 u, \ldots
\end{multline*}
which seems to have a period $p-1 = 4$, except for the first term.

For $p=7$, the similar computation gives
\begin{multline*}
  {\color{red} 2 u}, {\color{blue} 3 u, u^{3} + 6 u, 6 u^{3}, 4 u^{5} + 2 u, 2 u^{3} + 2 u, 6 u^{5} + 3 u^{3}}, 3 u , \ldots
\end{multline*}
with apparently a period $p-1 = 6$, except for the first term.

For $p=11$, one gets
\begin{multline*}
  {\color{red}7 u}, {\color{blue}8 u, 10 u^{3} + 4 u, 4 u^{3} + 7 u, 3 u^{5} + 3 u^{3} + 7 u, 7 u^{5} + 5 u^{3}, u^{7} + 10 u^{5} + 9 u,} \\ {\color{blue}7 u^{7} + 8 u^{3} + u, 2 u^{9} + 5 u^{5} + 2 u^{3} + 6 u, 9 u^{7} + u^{5} + 6 u^{3} + 5 u, u^{9} + 8 u^{7} + 10 u^{5} + 9 u^{3}}, 8 u, \ldots
\end{multline*}
with apparently a period $p-1 = 10$, except for the first term.

In general, for $p > 3$ prime, the same pattern seems to hold, with a
periodic sequence of period $p-1$ starting at the second term. This has
been checked on the first two such periods for all primes less than $500$.

One can note that the polynomials $p_m(u)$ have some resemblance with
the Bernoulli polynomials that are the values of Hurwitz zeta function
at negative integers.

The same periodicity of period $p-1$ seems to hold also for $p > 3$
and the pair $(\chi_4, X(X+u))$, where $\chi_4$ is the unique
primitive Dirichlet character of conductor $4$. One can certainly
expect this phenomemon to extend to other characters and periodic functions.



\bibliographystyle{plain}
\bibliography{Lfonction}

\verb?chapoton@unistra.fr?\\
Frédéric Chapoton\\
Institut de Recherche Mathématique Avancée,\\
UMR 7501 Université de Strasbourg et CNRS\\
7 rue René-Descartes, 67000 Strasbourg, France

\end{document}